\documentclass[12pt]{amsart}
\usepackage{geometry}
\usepackage{mathtools}
\usepackage{setspace}
\usepackage{hyperref}
\usepackage{amsmath, amssymb, amsthm}
\newcommand{\SET}{\mathrm{\mathbf{Set}}}
\usepackage{todonotes}
\newcommand{\GSET}{{\mathrm{\mathbf{G}}\mbox{-}\mathrm{\mathbf{Set}}}}

\newcommand{\xGSET}{(\mathrm{\mathbf{G}}\text{-}\mathrm{\mathbf{Set}})_{\mathbb{S}}}
\newcommand{\XGSET}{\mathbf{G}\text{-}\mathbb{XS}}
\usepackage[all,ps]{xy}
\newtheorem{thm}{Theorem}[section]
\newtheorem{prop}[thm]{Proposition}

\newtheorem{defi}[thm]{Definition}
\newtheorem{lem}[thm]{Lemma}
\newtheorem{ques}[thm]{Question}

\newcommand{\rC}{\mathrm C}
\newcommand{\rD}{\mathrm D}

\newcommand{\rF}{\mathrm F}
\newcommand{\rG}{\mathrm G}

\newcommand{\rK}{\mathrm K} 
\newcommand{\rM}{\mathrm M}
\newcommand{\rN}{\mathrm N}
\newcommand{\rQ}{\mathrm Q}

\newcommand{\rS}{\mathrm S}
\newcommand{\rT}{\mathrm T}
\newcommand{\rU}{\mathrm U}
\newcommand{\rV}{\mathrm V}

\newcommand{\co}{\colon}

\newcommand{\mc}{\mathcal}

\newcommand{\bC}{\mathbb C}

\newcommand{\bQ}{\mathbb Q}

\newcommand{\bS}{\mathbb S}
\newcommand{\bT}{\mathbb T}

\usepackage[numbers]{natbib}

\author{John Boiquaye}
\address{Department of Mathematics, 
	University of Ghana, 
	Botanical \newline \indent Gardens Road,
	Legon, GA-489-9348, 
	P.O. Box LG 62, Accra, 
	Ghana}
\email{jboiquaye@ug.edu.gh}

\author{Ralph Agyei Twum}
\address{Department of Mathematics, 
	University of Ghana, 
	Botanical \newline \indent Gardens Road,
	Legon, GA-489-9348, 
	P.O. Box LG 62, Accra, 
	Ghana}
\email{ratwum@ug.edu.gh}

\title{Duplicial Sets, crossd $G$-sets and Descent categories}

\begin{document}
	\maketitle
\begin{abstract}
The goal of this note is to show that the left $\chi$-coalgebra, which is an additional structure on one of the coefficients used in the construction of the cyclic operator for the cyclic sets that generalises the twisted nerve of a group by Loday is an instance of a general theory of left $\chi$-coalgebras. It is also shown that in the general case, the resulting cyclic operator  requires that the other coefficient needs to be equipped with a  crossed $G$-sets structure.	\end{abstract}	

\section{Introduction}

This paper is sequel to an earlier paper \cite{Boi25} where the first author studied the question ``when does a simplicial object that computes the simplicial homology of an algebraic structure carry a duplicial structure?". The motivating example is that of a canonical simplicial $\mathbb{F}$-module $C(U, M, N)$ which has as $n$-simplices the elements of $M \otimes U^{\otimes n} \otimes N$ and whose homology is $\mathrm{Tor}^{U/\mathbb{F}}(M, N)$ where $U$ is an associative unital $\mathbb{F}$-algebra and $M, N$ being right respectively left $U$-modules \cite[Section 8.7.1]{Weibel}.
This example has been well studied by many when $U$ is a Hopf algebra. From the perspective of the results of B\"ohm and \c Stefan \cite{BS08}, the author in \cite{Boi25} studied the above question in the category $\SET$ of sets where the Hopf algebra $U$ is just a group $G$ and $M, N$ are right respectively left $G$-sets. This led to the classification of all duplicial structures on the simplicial set $C(G, \{*\}, N)$ obtained from the bar resolution $\tilde{\rM}B_*(\bT, \rN) : \mathbf{1} \rightarrow \SET$ of the functor $\rN : \mathbf{1} \rightarrow \GSET$ (so that $\rN$ picks out an a $G$-set $N$) where $\bT \coloneq \rF\rU$ is the comonad on the category $\GSET$ of $G$-sets induced by the free-forgetful adjunction $\rF\coloneq \rG\tilde{\times} - : \SET \rightleftarrows \GSET: \rU $ and $\tilde{\rM} \colon \GSET \rightarrow \SET$ is a functor that sends any $G$-set $X$ to the set of its orbits $X/G$. In addition, the first author classified the  additional structures $\rho : \rT \rN \Rightarrow \rS \rN$ on the functor $\rN$ known as right $\chi$-coalgebra  where $\chi : \rT\rS \Rightarrow \rS\rT$ is a comonad distributive law \cite{beck}. The additional structures $\lambda : \tilde{\rM}S \Rightarrow \tilde{\rM}T$ on $\tilde{\rM}$ known as left $\chi$-coalgebras were also classified. These additional structures serve as coefficients in cyclic homology theory \cite{KKS15}.  Furthermore the twisted nerve of a group by Loday \cite{Loday} was generalized and shown to arise from the B\"ohm-\c Stefan construction \cite{BS08}.

 Kowalzig, Kr\"ahmer and Slevin in \cite[Proposition 3.5]{KKS15}, studied the above question from the perspective of B\"ohm and \c Stefan \cite{BS08} and gave a complete description of the right $\chi$-coalgebras $\rho$ on $\rN$. The results obtained by the first author in \cite[Proposition 5.4]{Boi25} is an instance of \cite[Proposition 3.5]{KKS15}. 
 However the description for left $\chi$-coalgebras $\lambda$ was not made explicit in \cite{KKS15} and although the first author classified these left $\chi$-coalgebra structures in the above mentioned setting in \cite[Proposition 5.5]{Boi25}, no description of these structures were provided. The aim of this paper is to give such analogous description of these structures. 
 
 In order to put our questions into perspective we recall the original setting in \cite{KKS15}. Assume 
 $\rF\co \mc A \rightleftarrows
 \mc B \co \rU$ is an adjunction for 
 $\bT$, meaning that 
 $\rT = \rF\rU$
 (see~\cite[Main
 Application 8.6.2]{Weibel} ). 
 Assume further that $\bS$ is a lift of a comonad 
 $\bC$ on~$\mc A$ through this
 adjunction. Then there is a  
 canonical distributive law 
 $ \chi \colon \rT \rS
 \Rightarrow \rS \rT$ which in the 
 
 case of an
 Eilenberg--Moore adjunction, 
  yields a
 bijective correspondence
 between lifts of comonads and
 distributive laws. 
 	\cite[Theorem
 	2.4]{Wol73}.  

 In \cite{BBK25} the author with two other coauthors provided an analogous description of left $\chi$-coalgebra structures on
 a functor $\rN \co \mc B \to
 \mc Z$. This description is provided in the following result:
 \begin{thm} \label{helpful} Considering the above setting, assume that
 	\begin{enumerate}
 		\item $\rV\co \mathcal{C} \rightleftarrows  \mathcal{B} \co \rG$ is an adjunction for the 
 		comonad~$\mathbb{S}$,
 		\item there is a comonad $\mathbb{Q}$ on~$\mc C$ which extends the comonad
 		$\mathbb{T}$ through the adjunction, and
 		\item the mate of the natural isomorphism~$\tilde \Omega \co \rQ \rG \Rightarrow \rG \rT$ which implements the extension is a lax isomorphism of comonads. 
 	\end{enumerate}
 	Then
 	\begin{enumerate}
 		\item the induced distributive law $\psi \co \bS\bT \Rightarrow \bT\bS$ is an isomorphism and 
 		\item left~$\psi^{-1}$-coalgebra structures~$(\rN, \lambda)$ on~$\mc Z$ correspond 
 		bijectively to~$\mathbb{Q}$-opcoalgebra structures on~$\rN\rV$.
 	\end{enumerate}
 \end{thm}

As mentioned above,  a classification result for left $\chi$-coalgebras was provided in \cite[Proposition 5.5]{Boi25}, which we capture below: 
	\begin{prop}\label{pro5} Let  $\tilde{\rM} : \GSET
		\rightarrow \SET$ be the
		functor which takes any
		$G$-set $X$ to the set of
		orbits in $X$, 
		\begin{align*} \tilde{\rM}X:=  X/G. \end{align*}
		\begin{enumerate} \item Let $h
			: L \rightarrow G$ be a
			$G$-equivariant map and let
			$\tilde{\lambda}_N : L \times
			N \rightarrow G\tilde{\times} N$ be a
			map defined by
			$\tilde{\lambda}_N(b, n) =
			(h(b), h(b)^{-1}n).$ Then
			$\tilde{\lambda}_N$ is a
			$G$-equivariant map.  \item
			The natural transformation $$\lambda_N : (L \times
			N)/G \rightarrow (G\tilde{\times}
			N)/G \quad ([b,n]) \mapsto
			[\tilde{\lambda}_N(b, n)]$$ 
			is a left
			$\chi$-coalgebra.
	\end{enumerate}
\end{prop}
	It is therefore natural to ask
	\begin{ques}
		Do the left $\chi$-coalgebras in Proposition \ref{pro5} apply to Theorem \ref{helpful}?
	\end{ques}
	\begin{ques}\label{q2}
		With reference to this example can we give a classification result for left $\chi$-coalgebras using Theorem \ref{helpful}? 
		If the answer to the above Question  is in the affirmative, what is the description of duplicial sets from this example, and when are these duplicial sets cyclic?  

	\end{ques}
	The content of our computation can be summarised as follows:
	\begin{thm}
	Let  $\tilde{\rM} : \GSET
	\rightarrow \SET$ be the
	functor which takes any
	$G$-set $X$ to the set of
	orbits in $X$, 
	$\tilde{\rM}X:=  X/G.$
	Let $N$ be a $G$-set and $h
	: L \rightarrow G$ be any
	$G$-equivariant map and let
	$\tilde{\lambda}_N : L \times
	N \rightarrow G\tilde{\times} N$ be the
	$G$-equivariant map defined by
	$\tilde{\lambda}_N(b, n) =
	(h(b), h(b)^{-1}n).$ 
	Then the left $\chi$-coalgebra $$\lambda_N : (L \times
	N)/G \rightarrow (G\tilde{\times}
	N)/G \quad ([b,n]) \mapsto
	[\tilde{\lambda}_N(b, n)]$$ 
	corresponds bijectively
	to $\bQ$-opcoalgebra structures on $\tilde{\rM}\rV$.
	\end{thm}
	\begin{thm}
		Assume that $N$ is a $G$-set and $L$ is a faithful $G$-set. The duplicial operator 
		$
		t : \bar{C}(G, \{*\}, N) \rightarrow \bar{C}(G, \{*\}, N) 
		$ such that for $n \in \mathbb{N}$
		$$ t_n([*, g_1, \ldots, g_{n+1}, w]) = [*, h(u), (h(u))^{-1}(g_1, \ldots, g_n), g_{n+1}w,]
		$$
		where $u = g_1\cdot\ldots\cdot g_{n+1}f(w)$
		is cyclic if and only if $(h\circ f(w))^{-1} \in \mathrm{Stab}(w)$ and $(h\circ f(gw))^{-1} = g(h\circ f(w))^{-1}g^{-1}$. 
	\end{thm}
	\section{Preliminaries and notation}\label{como}
	
	\subsection{The category $\GSET$}
	Throughout this paper, $G$
	denotes a (not necessarily
	finite) group (in the
	category $\SET$ of sets) with 
	unit element $1$. 
	By a $G$-set $N$, we mean a set
	with a left action 
	$ \theta \colon G \times N
	\rightarrow N $ that we
	usually just write as 
	concatenation, $
	gn := \theta(g,n) $. If $M$ is a set with a
	right action of $G$, we
	consider it also as a $G$-set 
	by setting $gm := m g^{-1}$, 
	$g \in G,m \in M$. The
	$G$-sets form a category 
	$\GSET$ 

	whose
	morphisms $L \rightarrow Y$ are
	the $G$-equivariant maps 
	$$ 
	\mathrm{Hom}_{\GSET}(L,Y) := 
	\{ f \colon L \rightarrow Y
	\mid 
	\forall  g \in G, x \in L :
	f(gx) = gf(x) 
	\}.
	$$
	
	\subsection{$\GSET$ is Cartesian}\label{comonads}
	The category $\GSET$ is
	Cartesian, i.e.~is a symmetric
	monoidal category in which
	every object is in a unique
	way a cocommutative comonoid.
	Concretely this means that
	given any $G$-sets $L,N$, the Cartesian
	product $L \times
	N$ of sets becomes a $G$-set with
	respect to the diagonal action
	$$
	g(l,n) := (gl,gn),
	$$
	and for a fixed $L$ this
	operation extends to an
	endofunctor $\rS := L \times -$ 
	on $\GSET$; furthermore, this
	endofunctor carries a unique
	(up to isomorphism) comonad
	structure which is given by
	the natural maps 
	$$
	\Delta _N \colon 
	L \times N \rightarrow 
	L \times L \times N,\quad
	(l,n) \mapsto 
	(l,l,n),
	$$ 
	$$
	\varepsilon_N \colon 
	L \times N \rightarrow N,
	\quad
	(l,n) \mapsto n.
	$$
	Throughout, we denote this comonad by
	$\mathbb{S}=(S,\Delta ,\varepsilon )$.
	
	\subsection{The comonad $\mathbb{T}$}\label{comonadt}
	Given any set $W$ and any $G$-set $L$, we can construct
	the $G$-set $\rF W:=L \tilde \times W$
	with action given by 
	$g(a,w):=(ga,w)$. 
	This extends in the obvious
	way to a functor 
	$\rF \colon \SET \rightarrow
	\GSET$. 
	When $L=G$, this 
	is left adjoint
	to the forgetful functor 
	$\rU \colon \GSET \rightarrow
	\SET$ that forgets the action
	of $G$. Hence we have a natural bijection 
	\begin{equation}
		\begin{aligned}\label{madina}
			\theta : \mathrm{Hom}_{\GSET}(\rF \rU N, \rS N) & \rightarrow \mathrm{Hom}_{\SET}(\rU N, \rU\rS N) 
			\\
			\theta(\rho)(n)&:= \rho(1, n) 
		\end{aligned}
	\end{equation}
	which has as inverse the map 
	\begin{equation}
		\begin{aligned}\label{santo}
			\xi : \mathrm{Hom}_{\SET}(\rU N, \rU\rS N) &\rightarrow  \mathrm{Hom}_{\GSET}(\rF \rU N, \rS N)\\
			\xi(\bar{\rho})(a, n)&: = a\bar{\rho}(n).
		\end{aligned}
	\end{equation}

	Now the composition 
	$\rT:=\rF\rU$ becomes (part of) a comonad 
	$\mathbb{T}=(\rT,\tilde \Delta ,\tilde \varepsilon )$ 
	on 
	$\GSET$ whose structure maps
	are given by
	$$
	\tilde{\Delta}_N \colon
	\rF N
	\rightarrow \rF\rF N
	,\quad
	(g,n) \mapsto 
	(g,1,n),
	$$	
	$$
	\tilde{\varepsilon}_N \colon
	\rF N \rightarrow N,
	\quad
	(g,n) \mapsto gn.
	$$
	
	\subsection{The category $\XGSET$}

	By a crossed $G$-set \cite{FreydYetter, kassel} $X$ we mean a pair $(X, \alpha^X)$ where $X$ is a $G$-set, and $\alpha^X : X \rightarrow G$ is set map such that for $g \in G$ and $x \in X$, $\alpha^X(gx) = g \alpha^X(x)g^{-1}$.
	Given two crossed $G$-sets $(Y, \alpha^Y)$ and $(X,\alpha^X)$, a map $f \colon Y \rightarrow X$ of crossed $G$-sets is a $G$-equivariant map  such that $\alpha^X \circ f = \alpha^Y$.  Crossed $G$-sets together with morphisms between crossed $G$-sets form a category which we will denote by $\XGSET$ \cite{FreydYetter}. 
	\subsection{Examples of crossed $G$-sets}
	It is well known that every $G$-set $X$ can be turned into crossed $G$-set by defining a map $\alpha^X : X \rightarrow G$ such that $\alpha^X(x) = 1$ for every $x \in X$. However if $X$ is a trivial $G$-set, then by setting $\alpha^X(x) \in C(G)$, where $C(G)$ is the center of $G$, we make the trivial $G$-set $X$, into a crossed $G$-set.
	In particular if $X = G$ as a $G$-set with action given by conjugation, then $G$ is a crossed $G$-set where $\alpha^G(g) = g$ for every $g \in G$, that is, $\alpha^G = \mathrm{id}$. However if the action is given by the group product then $G$ is a crossed $G$-set with $\alpha(g) = 1$ for every $g \in G$; that is, $\alpha$ is the trivial map \cite{kassel,Yoshida, FreydYetter}.
	\subsection{$\XGSET$ is braided monoidal}
	The category of crossed $G$-sets is a braided monoidal category \cite{FreydYetter}, in which every object is a unique comonoid. Given any crossed $G$-sets $Y, X$, the product of $X$ and $Y$ is given by the cartesian product~$X \times Y$. The product $X \times Y$ is turned into a crossed $G$-set with respect to the diagonal action
	$$
	g(x, y):= (gx, gy)
	$$ of $G$ and the map $$\alpha^{X \times Y} (x, y) = \alpha^X(x)\alpha^Y(y).$$
	As a braided category, the braiding $\sigma_{X, Y}: X \times Y \rightarrow Y\times X$ is given by~$$\sigma_{X, Y}(x, y) = (y, (\alpha^Y(y))x).$$

	\subsection{The category $\xGSET$}
	Here we discuss coalgebras over the comonad $\mathbb{S}$, discussed in Section \ref{comonads}. These are pairs $(X, \beta)$ where $X$ is a $G$-set and $\beta : X \rightarrow \rS X$ a map satisfying the following equations:
	$$
	\varepsilon \circ \beta = \mathrm{id}, \quad \Delta Y \circ \beta = \rS\beta \circ \beta
	$$
	from which we obtain the condition that $\beta(x) = (\beta_1(x), x)$ for any $x \in X$, where $\beta_1 : X \rightarrow L$ is a $G$-equivariant map. Thus $\beta$ is determined by $\beta_1$. Given $(Y, \alpha)$  another coalgebra over the comonad $\mathbb{S}$, a morphism $h : (X, \beta) \rightarrow (Y, \alpha)$ is a $G$-equivariant map $h: X \rightarrow Y$ satisfying the following equations:
	$$
	\rS h \circ \beta = \alpha \circ h.
	$$
	Coalgebras over the comonad $\mathbb{S}$ together with morphisms between them form a category which we denote by $\xGSET$. As $\xGSET$ is a an Eilenberg-Moore category, we know that it stems from and Eilenberg-Moore adjunction with the co-free functor $\rF^{\mathbb{S}}$ being right adjoint to the forgetful functor $V$ where $\mathrm{F}^{\mathbb{S}} : \GSET \rightarrow \xGSET$, $ X \mapsto (\rS X, \Delta)$ and $ \rV : \xGSET \rightarrow \GSET$, $(Y, \alpha : Y \rightarrow \rS Y ) \mapsto Y$.  Explicitly, we have a natural isomorphism 
	\begin{align*}
	\xi : \mathrm{Hom}_{\xGSET}((Y, \alpha), \mathrm{F}^{\mathbb S}(X)) &\rightarrow \mathrm{Hom}_{\GSET} (\rV (Y, \alpha), X) \\	
	\xi(\mathfrak{f})(y)&:= (\varepsilon\circ \mathfrak{f})(y)
	\end{align*}
	
	with inverse 
	\begin{align*}
		\Theta : \mathrm{Hom}_{\GSET}(\rV (Y, \alpha), X) &\rightarrow \mathrm{Hom}_{\xGSET}((Y, \alpha), \mathrm{F}^{\mathbb{S}}(X))\\
		& \Theta(\mathfrak{h})(y):= (\rF^{\mathbb{S}}(h) \circ \alpha) (y).
		\end{align*}
	The unit  $\eta^{\rF^{\mathbb{S}}\rV}_{(Y, \alpha)}: (Y, \alpha) \rightarrow \mathrm{F}^\mathbb{S}\rV (Y, \alpha)$ of the adjunction is given by the $\mathbb{S}$-coalgebra structure map $\alpha$,  while the counit $\varepsilon^{\rV\rF^{\mathbb{S}}}_X : \rV\mathrm{F}^{\mathbb{S}}X \rightarrow X$ of the adjunction is given by the counit $\varepsilon$ of the comonad $\mathbb{S}$.

	\section{Adjunctions, lifts and Distributive laws}
Here we recall some definitions that we will need in our main result. 
Suppose that $\rF\colon \mathcal{A} \rightleftarrows \mathcal{B} \colon \rU
	$ is an adjunction with unit $\eta$ and counit $\varepsilon$. 
	
	\begin{defi}
		A lift of an endofunctor $\rC$ through the adjunction $\rF \dashv \rU$ is an endofunctor $\rS$ together with a natural isomorphism $ \Omega : \rC\rU \Rightarrow \rU\rS$. In such a setup, $\rC$ is also referred to as the extension of $\rS$ through the adjunction and the following diagram
		$$
		\xymatrix{ \mathcal{B} \ar[d]_{\rU}\ar[r]^{\rS} & \mathcal{B} \ar[d]^{\rU} \\ \mathcal{A} \ar[r]_{\rC} & \mathcal{A}}
		$$
		commutes up to the natural isomorphism $\Omega : \rC\rU \Rightarrow \rU\rS$. 
		\end{defi}
		Such a natural transformation   $\Omega : \rC\rU \Rightarrow \rU\rS$ uniquely determines  another natural transformation $\Lambda : \rF\rC \Rightarrow \rS\rF$ which is its mate where $\Lambda = \varepsilon \rS\rF \circ \rF \Omega \rF \circ \rF \eta$. As pointed out in \cite[Section 2.6,~ Example 2.18 ]{BBK25} $\Lambda$ is not in general invertible even when $\Omega$ is invertible.

	We now consider the following setup in \cite{KKS15} with $\mathcal{A}\coloneq \SET$,  $\mathcal{B}\coloneq \GSET$ and 
	$\mathcal{C}\coloneq \xGSET$ with adjunctions 
	$\rF \coloneq G\tilde{\times} - : \SET \rightleftarrows \GSET \colon \mathrm{U}$ and $ \rV\colon \xGSET \rightleftarrows \GSET :\rF^{\mathbb{S}} $. The adjunction $\rF \coloneq G\tilde{\times} - : \SET \rightleftarrows \GSET \colon \mathrm{U}$  induces a comonad $\mathbb{T}$ (see Section \ref{comonadt}) and taking $C$ to be a comonad on $\SET$ with its lift the comonad $S$ (see Section \ref{comonads}) we have the lower half of the diagram below:
	$$
		\xymatrix{\xGSET \ar[rr]^\rQ\ar@/_1.2pc/[dd]_{\rV} &  & \xGSET \ar@/_1.2pc/[dd]_{\rV}\\ \\
				\GSET\ar@/^1.2pc/[dd]^{\rU }\ar@{}[dd]|{\dashv} \ar@/_1.2pc/[uu]_{\rF^{\mathbb{S}}}\ar@{}[uu]|{\dashv} \ar@<0.4ex>[rr]^{\rT} \ar@<-0.4ex>[rr]_{\rS = L \times -} & & \GSET \ar@/^1.2pc/[dd]^{\rU}\ar@{}[dd]|{\dashv} \ar@/_1.2pc/[uu]_{\rF^{\mathbb{S}}} \ar@{}[uu]|{\dashv}\\ \\
				\SET\ar@/^1.2pc/[uu]^{\rG\tilde{\times}-} \ar[rr]_\rC & & \SET \ar@/^1.2pc/[uu]^{\rG\tilde{\times}-}
			} 
		$$
		
		The lower half of the above diagram was indirectly considered in \cite{Boi25}. We concentrate our discussion on the upper half of the diagram. 
	In what follows we recall the comonad distributive law in \cite[Section 5, Proposition 5.3]{Boi25} and show that it is invertible. Let $\mathbb{T, S}$ be the comonads in Section \ref{como}.
		The natural transformation 
	$ \chi : \rT\rS \Rightarrow \rS\rT$ given by
	$\chi_N: G\tilde{\times}L\times X \rightarrow
	L\times G\tilde{\times} X \quad (g, l, x) \mapsto (gl, g, x)$ 
	is the unique distributive law between the comonads  
	$\mathbb{T}$ and $\mathbb{S}$.
	\begin{lem}
	Let $X$ be any $G$-set. Then the map $$\chi_X : G\tilde{\times}L \times X \rightarrow L \times G\tilde{\times}X \quad (g, l, x) \mapsto (gl, g, x)$$ is invertible with inverse map given by 
	$$\tilde{\chi}_X: L\times G\tilde{\times} X \rightarrow G\tilde{\times}L \times X \quad (l,g,x) \mapsto (g, g^{-1}l, x).
	$$
	Furthermore the inverse map is $G$-equivariant. 
		\end{lem}
	\begin{proof}
	Since 
	\begin{eqnarray*}\tilde{\chi}_X\circ \chi_X(g, l, x) &=&\tilde{\chi}_X(gl, g, x)\\ &= & (g, g^{-1}gl, x)= \mathrm{Id}_{G\tilde{\times} L \times X}(g, l,x)
	\end{eqnarray*}
and 
\begin{eqnarray*}
	\chi_X\circ \tilde{\chi}_X(l, g, x) &=&\chi_X(g, g^{-1}l, x) \\&=&(gg^{-1}l,g, x) = \mathrm{Id}_{L \times G\tilde{\times} X}(g, l, x)
\end{eqnarray*}
it implies that $\chi_X$ is invertible. That $\tilde{\chi}_X$ is $G$-equivariant is a straightforward computation. 
	\end{proof}
	
\begin{prop}
The natural transformation $\tilde{\chi}: \rS\rT \Rightarrow \rT\rS$ is the unique distributive law between the comonads $\mathbb{S}$ and $\mathbb{T}$.	
	\end{prop}	
	\begin{proof}
	It is straightforward to show that $\tilde{\chi}$ is a distributive law, i.e, that the following diagrams 
	$$
	\xymatrix{L \times G\tilde{\times}
		X  \ar[d]_{\tilde{\chi}_X}
		\ar[r]^{L\tilde{\times} \tilde{\Delta}_X \;\;\;\;}&
		L\times G \tilde{\times} G\tilde{\times}
		X\ar[r]^{\tilde{\chi}_{G\tilde{\times} X}}  &
		G\tilde{\times} L \times G\tilde{\times}
		X\ar[d]^{\;\;G\tilde{\times} \tilde{\chi}_X}
		\\ G \tilde{\times} L \times
		X\ar[rr]_{\tilde{\Delta}_{L\times
				X}} & & G\tilde{\times} G\tilde{\times} L
			\times X } \qquad
	\xymatrix{L\times G\tilde{\times} X
		\ar[r]^{\tilde{\chi_X}}\ar[dr]_{L\times
			\tilde{\varepsilon}_X} & G\tilde{\times}L \times
		X\ar[d]^{\tilde{\varepsilon}_{L
					\times X}} \\ &  L \times X}
	$$
	
	 $$
	 \xymatrix{ L\times G\tilde{\times}
	 X  \ar[d]_{\tilde{\chi}_X}
		\ar[r]^{\Delta_{G\tilde{\times
				X}\;\;\;\;\;\;}}  & L\times L \times G\tilde{
		\times} X\ar[r]^{L\times
			\tilde{\chi}_X}  & L\times G\tilde{\times}
		L \times X\ar[d]^{\;\; \tilde{\chi}_{L
				\times X}} \\ G\tilde{\times} L
		\times X\ar[rr]_{G\tilde{\times}\Delta_X} & & G\tilde{\times}
		L \times L \times X } \qquad
	\xymatrix{L\times G \tilde{\times} X
		\ar[r]^{\tilde{\chi}_X}\ar[dr]_{\varepsilon_{G\tilde{
				\times} X}} & G \tilde{\times} L \times
		X\ar[d]^{G \tilde{\times}
			\varepsilon_X} \\ &  G\tilde{
		\times X.}} 
	$$
		commute. Assume now that $\tilde{\chi}_X : L \times G\tilde{\times}X \rightarrow G\tilde{\times} L \times X$ is any distributive law between the comonads $\mathbb{S}$ and $\mathbb{T}$, define $\beta, \alpha, \xi$ by $\tilde{\chi}_X(l,g, x)\eqcolon (\beta(l, g, x), \alpha(l, g, x), \xi(l, g, x))$. We obtain from the first counit compatibility condition that
		$$
		\tilde{\varepsilon}_{L \times X}\circ \tilde{\chi}_X(l, g, x) = L\times \tilde{\varepsilon}(l, g, x)
		$$ 
		So that 
		$$ (\beta(l, g,x)\alpha(l,g,x), \beta(l, g, x)\xi(l,g,x)) = (l, gx).$$ 
		Thus $\beta(l,g,x)\alpha(l,g,x) = l$ and $\beta(l,g,x)\xi(l,g,x) = gx$. From the second counit compatibility condition we have 
		$$
		G \tilde{\times}\varepsilon_X\circ \tilde{\chi}(l,g,x) = \varepsilon_{G\tilde{\times} X}(l,g,x).
		$$
		That is 
		$$
		(g, x) = (\beta(l,g,x), \xi(l, g,x)).
		$$
		Therefore $\beta(l, g,x) = g$, $\xi(l,g,x) = x$ and $\alpha(l,g,x)= g^{-1}l$.
	\end{proof}
	
	\subsection{Extension of the comonad $\mathbb{T}$}
	Here we assume that the comonad $\mathbb{T}$ is a lift of a comonad $\mathbb{Q} = (\rQ, \Delta^{\mathbb{Q}}, \varepsilon^{\mathbb{Q}})$, or $\mathbb{Q}$ is an extension of $\mathbb{T}$ through the adjunction $ \rV\colon \xGSET \rightleftarrows \GSET :\rF^{\mathbb{S}} $ and we write out explicitly what $\mathbb{Q}$ should be. That is we assume we have a natural isomorphism $\tilde{\Omega} :  \rQ\rF^{\mathbb{S}} \Rightarrow \rF^{\mathbb{S}}\rT$. 
	\begin{prop}
		Let $\chi : \rT\rS \Rightarrow \rS\rT$ be the comonad distributive law between the comonads $\mathbb{T}$ and $\mathbb{S}$. If $\rQ: \xGSET \rightarrow \xGSET$ is an endofunctor on $\xGSET$ such that for any $\mathbb{S}$-coalgebra pair $(X, \beta),$ 
		$
		\rQ(X, \beta) = (G\tilde{\times}X, \chi \circ G\tilde{\times}\beta)
		$ then $\rQ$ is part of a comonad $\bQ$ with natural transformations $\Delta^{\mathbb{Q}}\colon \rQ \Rightarrow \rQ\rQ$ with
		$$\Delta^\bQ_{(X, \beta)}\colon (G\tilde{\times}X, ~\chi\circ G\tilde{\times} \beta) \rightarrow (G\tilde{\times}G\tilde{\times}X, ~\chi_{G\tilde{\times}X} \circ G \tilde{\times} \chi_X \circ G \tilde{\times}G\tilde{\times}\beta)
		$$
		$$
		\Delta^\bQ_{(X, \beta)}(g, x) = \tilde{\Delta}_X(g,x) = (g, 1, x)
		$$
		and $\varepsilon^\bQ : \rQ \Rightarrow \mathrm{Id}_{\xGSET}$
	with
	$\varepsilon^\bQ_{(X, \beta)}: (G\tilde{\times}X, \chi \circ G\tilde{\times}\beta) \rightarrow (X, \beta)
	$ defined by  $$\varepsilon^\bQ{(X, \beta)}(g, x) = \tilde{\varepsilon}_X(g, x) = gx$$.
		\end{prop}
	\begin{proof}
	That $\Delta^\bQ$ is coassociative is straightforward as $\Delta^\bQ_{(X, \beta)} = \tilde\Delta_X$. It remains to show that $\Delta^{\bQ}_{(X,\beta)}$ is a morphism of $\mathbb{S}$-coalgebras, that is 
	$$
	L \times \Delta^\bQ_{(X, \beta)} \circ \chi_X\circ G\tilde{\times}\beta = \chi_{G\tilde{\times}X} \circ G\tilde{\times}\chi_X \circ G \tilde{\times}G\tilde{\times}\beta \circ \Delta^\bQ_{(X, \beta)}.
	$$	
	Now 
	\begin{eqnarray*}L \times \Delta^\bQ_{(X, \beta)} \circ \chi_X \circ G\tilde{\times}\beta(g, x)&=& L \times \Delta^\bQ_{(X, \beta)}\circ \chi_X(g, \beta_1(x), x)\\ &=& L \times \Delta^\bQ_{(X, \beta})(g\beta_1( x), g, x) \\&=& (g\beta_1(x), g, 1, x) 
\end{eqnarray*}
On the other hand, 
\begin{eqnarray*}
	\chi_{G\tilde{\times}X} \circ G\tilde{\times}\chi_X \circ G \tilde{\times}G\tilde{\times}\beta \circ \Delta^\bQ_{(X, \beta)}(g, x) &=& \chi_{G\tilde{\times}X} \circ G\tilde{\times}\chi_X \circ G \tilde{\times}G\tilde{\times}\beta(g, 1, x)\\ &=& \chi_{G\tilde{\times}X} \circ G\tilde{\times}\chi_X((g, 1, \beta_1(x), x)) \\&=& \chi_{G\tilde{\times}X}(g, \beta_1(x), 1, x) \\&=& (g \beta_1(x),g, 1, x)
\end{eqnarray*}
Also, that $\varepsilon^\bQ_{(X, \beta)}$ satisfies the counitality condition is straightforward as $\varepsilon^\bQ_{(X, \beta)} = \tilde\varepsilon_X$. It remains to show that $\bar\varepsilon_{(X, \beta)}$ is a morphism of $\mathbb{S}$-coalgebras, that is,
$$
\beta \circ \varepsilon^\bQ_{(X, \beta)} = L \times \varepsilon^\bQ_{(X, \beta)} \circ \chi_X \circ G\tilde{\times}\beta.
$$
We have that 
$\beta \circ \varepsilon^\bQ_{(X, \beta)}(g, x) = (\beta_1(gx), gx) = (g\beta_1(x), gx)$. On the other hand, 
\begin{eqnarray*}
	L \times \varepsilon^\bQ_{(X, \beta)} \circ \chi_X \circ G\tilde{\times}\beta(g,x) &=& L \times \varepsilon^\bQ_{(X, \beta)}(g\beta_1(x), g, x) \\&=&
	(g\beta_1(x), gx)
\end{eqnarray*}
Thus $\mathbb{Q} = (\rQ, \Delta^\bQ, \varepsilon^\bQ)$ is a comonad on $\xGSET$.
	\end{proof}
	Next we classify all natural transformations $ \tilde{\Omega}\colon\rQ\rF^{\mathbb{S}} \Rightarrow \rF^{\mathbb{S}}\rT$.
	
	\begin{lem}\label{lem32}
		Let $X$ be a $G$-set and let $\bar{a}$ be a fixed element in $G$. Then
		 \begin{align*}
		 	\tilde\Omega_X : G\tilde{\times} L \times X &\rightarrow L \times G\tilde{\times}X \\ 	
		(g, l, x) &\mapsto (gl, g\bar{a}, x)
		\end{align*}
	is an invertible $G$-equivariant map.
	\end{lem}
	\begin{proof}
That $\tilde{\Omega}_X$ is $G$-equivariant is a straightforward calculation. It is invertible with inverse 
\begin{align*}
\Gamma_X : L \times G\tilde{\times}X &\rightarrow G \tilde{\times} L \times X \\ (l, g, x) &\mapsto (g\bar{a}^{-1}, \bar{a}g^{-1}l, x).
\end{align*}
	We see that  	
\begin{align*}
\Gamma_X \circ \tilde{\Omega}_X(g, l, x)&= \Gamma (gl, g\bar{a}, x) \\&= (g, l, x)  \\ \\
\tilde{\Omega}_X \circ \Gamma_X(l, g,x) &= \tilde{\Omega}_X(g\bar{a}^{-1}, \bar{a}g^{-1}l, x) \\ &= (l, g, x).
\end{align*}
So $ \Gamma_X \circ \tilde{\Omega}_X = \mathrm{Id}_{G\tilde{\times}L \times X}$ and $\tilde{\Omega}_X \circ \Gamma_X = \mathrm{Id}_{L \times G\tilde{\times} X}$.
	\end{proof}
	
	\begin{lem}\label{lem33}
		Let $X$ be a $G$-set, and let $(G\tilde{\times}L \times X, \chi_X \circ G\tilde{\times}\Delta_X)$,  $(L \times G \tilde{\times} X, \Delta_{G\tilde{\times}X})$ be $\mathbb{S}$-coalgebras.
		The maps $\tilde{\Omega}_X : (G\tilde{\times} L \times X, \chi_{L\times X}\circ G\tilde{\times} \Delta_X) \rightarrow (L \times G\tilde{\times} X, \Delta_{G\tilde{\times}X})$ and $\Gamma_X \colon (L \times G \tilde{\times} X, \Delta_{G\tilde{\times}X}) \rightarrow (G\tilde{\times}L \times X, \chi_X \circ G\tilde{\times}\Delta_X)$  are morphisms of $\mathbb{S}$-coalgebras.

	\end{lem}
	\begin{proof}
	It is a straightforward computation to show that $$\tilde{\Omega}_X : (G\tilde{\times} L \times X, \chi_{L\times X}\circ G\tilde{\times} \Delta_X) \rightarrow (L \times G\tilde{\times} X, \Delta_{G\tilde{\times}X})$$ is a morphism of $\bS$-coalgebras, that is,
	the equation 
$$
		\Delta_{G\tilde{\times}X} \circ \tilde{\Omega}_X = L \times \tilde{\Omega}_X \circ \chi_{L\times X} \circ G\tilde{\times} \Delta_X
	$$ holds. Similarly, it is a straightforward computation to show that $$\Gamma_X \colon (L \times G \tilde{\times} X, \Delta_{G\tilde{\times}X}) \rightarrow (G\tilde{\times}L \times X, \chi_{L\times X} \circ G\tilde{\times}\Delta_X)$$ 
	is a morphism of $\bS$-coalgebras, that is, the following equation
	$$
	\chi_{L\times X} \circ G\tilde{\times} \Delta_X \circ \Gamma_X = L \times \Gamma_X \circ \Delta_{G\tilde{\times} X}
	$$ 
	holds.
	
	\end{proof}
	\begin{prop} \label{prop33}
		The pair $(\rF^{\mathbb{S}}, \tilde{\Omega})$ is a lax isomorphism from the comonad $\mathbb{Q}$ on $\xGSET$ to the comonad $\mathbb{T}$ on $\GSET$ if and only if $\bar{a} = 1$.
	\end{prop}
	\begin{proof}
	Suppose that $(\rF^{\mathbb{S}}, \tilde{\Omega})$ is lax isomorphism from $\mathbb{Q}$ to $\mathbb{T}$. Then $\tilde{\Omega} : \rQ\rF^{\mathbb{S}} \Rightarrow \rF^{\mathbb{S}}\rT$ is a natural isomorphism. That is for every $G$-set $X$, $$\tilde{\Omega}_X : (G\tilde{\times} L \times X, \chi_{L\times X} \circ G\tilde{\times}\Delta_X) \rightarrow (L \times G \tilde{\times}X, \Delta_{G\tilde{\times} X})$$
	 is an isomorphism of $\mathbb{S}$-coalgebras. In addition the following diagrams 
	 $$\xymatrix{ G\tilde{\times} L \times X \ar[d]_{\tilde{\Omega}_X}\ar[r]^{\Delta^{\mathbb{Q}}_{L \times X}} & G\tilde{\times} G\tilde{\times} L \times X \ar[r]^{G \tilde{\times} \tilde{\Omega}_X} & G\tilde{\times}L \times G\tilde{\times} X \ar[d]^{\tilde{\Omega}_{G\tilde{\times}X}} \\ L \times G\tilde{\times} X \ar[rr]_{ L \times \Delta^{\mathbb{Q}}_X} & & L \times G\tilde{\times}G\tilde{\times} X} \qquad \xymatrix{ G\tilde{\times} L \times X \ar[dr]_{\varepsilon^{\bQ}_{L\times X}} \ar[r]^{\tilde{\Omega}_X}  &  L\times G\tilde{\times}X \ar[d]^{L\times \varepsilon^{\bQ}_X} \\ & L \times X   }
	 $$
	 commute. From the first diagram above, we obtain
	 \begin{align*}
	 \tilde{\Omega}_{G\tilde{\times}X} \circ G \tilde{\times} \tilde{\Omega}_X \circ \Delta^{\bQ}_{L \times X}(g, l, x) &= L \times \Delta^{\bQ}_X \circ \tilde{\Omega}_X(g, l, x) \\ (gl,g\bar{a},\bar{a}, x)&= (gl, g\bar{a}, 1, x)
	 \end{align*}
	 Hence $\bar{a} = 1$. Conversely, suppose that $\bar{a} = 1$ and let  $\tilde{\Omega}\colon \rQ\rF^{\bS} \Rightarrow \rF^{\bS} \rT$ be a natural transformation. Then for $X$ a $G$-set we obtain 
	 $$\tilde{\Omega}_X : (G\tilde{\times} L \times X, \chi_{L\times X} \circ G\tilde{\times}\Delta_X) \rightarrow (L \times G \tilde{\times}X, \Delta_{G\tilde{\times} X})$$
	 which is an isomorphism of $\bS$-coalgebras by Lemmata \ref{lem32} and \ref{lem33}.
	 It remains to show that the two diagrams above commute. We see clearly from the above computation that if $\bar{a} = 1$, then the rectangular diagram above commutes. With the triangular diagram above, 
	 $$
	 L \times \varepsilon^{\bQ}_X \circ \tilde{\Omega}_X(g, l, x) = (gl, gx) = \varepsilon^{\bQ}_{L \times X}(g, l, x).
	 $$
	\end{proof}
	\begin{prop}
		The pair $(\tilde{\Omega}, \mathbb{T})$ is a lift of the comonad $\mathbb{Q}$ under the adjunction $V : \xGSET \rightleftarrows \GSET$, i.e, 
		$
		\tilde\Omega : \rQ\rF^{\mathbb{S}} \Rightarrow \rF^{\mathbb{S}}\rT 
		$ is a lax isomorphism of comonads.
		\end{prop}
		\begin{proof}
		From Lemmata \ref{lem32} and \ref{lem33} and Proposition \ref{prop33} we have that $\tilde{\Omega}$ is a lax isomorphism.
		\end{proof}
		Since $\tilde{\Omega}$ uniquely induces a mate, we have $\tilde{\Lambda} : \rV\rQ \Rightarrow \rT\rV$ by the following formula 
		$$
		\tilde{\Lambda} \coloneq \varepsilon^{VF^{\mathbb{S}}}\rQ\rV \circ \rV\tilde{\Omega}\rV \circ \rV\rQ\eta^{F^{\mathbb{S}}V} : \rV\rQ \Rightarrow \rT\rV. 
		$$
		Explicitly, for any $\mathbb{S}$-coalgebra $(X, \beta)$, 
		$\tilde{\Lambda}_{(X, \beta)} : G\tilde{\times} X \rightarrow G \tilde{\times} X$ sends $(g, x)$ to $(g\bar{a}, x)$.
		As already mentioned, the mate induced by a lift through an adjunction is not invertible in general although the lift is. In what follows we see a peculiar case where the mate $\tilde{\Lambda}$ is invertible.
\begin{prop}\label{prop35}
 The natural transformation $\tilde{\Lambda} : \rV\rQ \Rightarrow \rT\rV$ is invertible. 
\end{prop}
\begin{proof}
 For any $\mathbb{S}$-coalgebra $(X, \beta)$, \begin{align*}\tilde{\Lambda}_{(X, \beta)} : G\tilde{\times} X &\rightarrow G\tilde{\times} X \\ (g, x) &\mapsto (g\bar{a}, x)
 	\end{align*}
 	has an inverse 
 	\begin{align*}\tilde{\Gamma}_{(X, \beta)} : G\tilde{\times} X &\rightarrow G\tilde{\times} X \\ (g, x) &\mapsto (g\bar{a}^{-1}, x)
 		\end{align*}
 		It is a straightforward calculation to show that $\tilde{\Omega}_{(X, \beta)} \circ \tilde{\Gamma}_{(X, \beta)} = \mathrm{Id}_{G\tilde{\times}X}$ and $\tilde{\Gamma}_{(X, \beta)} \circ \tilde{\Omega}_{(X, \beta)} = \mathrm{Id}_{G\tilde{\times}X}$.
\end{proof}
\begin{lem}
	The pair $(\rV, \tilde{\Lambda})$ is colax isomorphism from the comonad $\mathbb{Q}$ on $\xGSET$ to the comonad $\mathbb{T}$ on $\GSET$ if and only if $\bar{a} = 1$.
\end{lem}
\begin{proof}
	Suppose that $(V, \tilde{\Lambda})$ is a colax isomorphism from the comonad $\bQ$ to the comonad $\bT$. Then $\tilde{\Lambda} : \rV\rQ \Rightarrow \rT\rV$ is a natural isomorphism. That is for any $\bS$-coalgebra $(X, \beta)$,
	$
	\tilde{\Lambda}_{(X, \beta)} \colon G\tilde{\times} X \rightarrow G\tilde{\times}X
	$ is a bijective $G$-equivariant map and the following diagrams 
	 $$\xymatrix{ G\tilde{\times} X \ar[d]_{\tilde{\Lambda}_X}\ar[r]^{\Delta^{\mathbb{Q}}_{X}} & G\tilde{\times} G\tilde{\times} X \ar[r]^{G \tilde{\times} \tilde{\Lambda}_X} & G\tilde{\times} G\tilde{\times} X \ar[d]^{\tilde{\Lambda}_{G\tilde{\times}X}} \\  G\tilde{\times} X \ar[rr]_{ \Delta^{\mathbb{Q}}_X} & &  G\tilde{\times}G\tilde{\times} X} \qquad \xymatrix{ G\tilde{\times} X \ar[dr]_{\varepsilon^{\bQ}_{ X}} \ar[r]^{\tilde{\Lambda}_X}  &  G\tilde{\times}X \ar[d]^{ \varepsilon^{\bQ}_X} \\ & X   }
	$$
	commute.
	From the first diagram above we have that 
	\begin{align*}
		\tilde{\Lambda}_{G\tilde{\times} X} \circ G\tilde{\times}\tilde{\Lambda}_X \circ \Delta^{\bQ}(g, x) &= \Delta^{\bQ}_X \circ \tilde{\Lambda}_X(g, x) \\ (g\bar{a}, 1, x) & = (g\bar{a}, \bar{a}, x)
	\end{align*}
	Thus $\bar{a} = 1$. Conversely, suppose that $\bar{a} = 1$ and let $\tilde{\Lambda}_{(X, \beta)} \colon G\tilde{\times}X \rightarrow G\tilde{\times}X$. It is straightforward to show that this map is $G$-equivariant. Next by Proposition \ref{prop35} $\tilde{\Lambda}$ is invertible. Furthermore we see that the above diagrams commute if $\bar{a} = 1$. Thus the natural transformation $\tilde{\Lambda} : \rV\rQ \Rightarrow \rT\rV$ is a colax isomorphism of comonads. Thus $\tilde{\Lambda}_{(X, \beta)}$ is the identity map for every $\bS$-coalgebra $(X, \beta)$.
\end{proof}

\section{Left $\chi$-coalgebras and $\mathbb Q$-opcoalgebras}
Here we present the main results of the work. We recall the classification result of left $\chi$-coalgebras in \cite[Proposition 5.5]{Boi25} and show that it is an instance of \cite[Theorem 3.1]{BBK25}. For the sake of completeness we state the Proposition without proof:
	\begin{prop}\label{prop5} Let  $\tilde{M} : \GSET
		\rightarrow \SET$ be the
		functor which takes any
		$G$-set $X$ to the set of
		orbits in $X$, 
		\begin{align*} \tilde{\rM}X:=  X/G. \end{align*}
		\begin{enumerate} \item Let $h
			: L \rightarrow G$ be a
			$G$-equivariant map and let
			$\tilde{\lambda}_N : L \times
			N \rightarrow G\tilde{\times} N$ be a
			map defined by
			$\tilde{\lambda}_N(b, n) =
			(h(b), h(b)^{-1}n).$ Then
			$\tilde{\lambda}_N$ is a
			$G$-equivariant map.  \item
			The natural transformation $$\lambda_N : (L \times
			N)/G \rightarrow (G\tilde{\times}
			N)/G \quad ([b,n]) \mapsto
			[\tilde{\lambda}_N(b, n)]$$ 
			is a left
			$\chi$-coalgebra.
	\end{enumerate}
	\end{prop}


\begin{prop}\label{corresp}
	Let $(X, \beta)$ be an $\mathbb{S}$-coalgebra and $h : L \rightarrow G$ be a $G$-equivariant map and consider the $\mathbb{S}$-coalgebra $\rQ(X, \beta) = (G\tilde{\times}X, \chi_X \circ G \tilde{\times} \beta)$.
	\begin{enumerate}
		\item The natural transformation $\nabla : \mathrm{Id}_{\xGSET} \Rightarrow \rQ$ has its components
	$$\nabla_{(X, \beta)} : (X, \beta) \rightarrow (G\tilde{\times}X, \chi_X\circ G\tilde{\times} \beta)  \qquad x \mapsto (h(\beta_1(x)), h(\beta_1(x))^{-1}x)$$
	 as morphisms of $\mathbb{S}$-coalgebras. 
	\item The pair $((X, \beta), \nabla_{(X, \beta)})$ is a $\mathbb{Q}$-coalgebra and $(X, \rV\nabla_{(X, \beta)})$ is a $\rV\mathbb{Q}$-coalgebra.
	\item The pair $(\tilde{\rM}X, \tilde{\rM}\rV\nabla_{(X, \beta)})$ is an $\tilde{\rM}\rV\bQ$-coalgebra.
	\end{enumerate}
\end{prop}

\begin{proof}
It is straightforward to show that $\nabla_{(X, \beta)} : X \rightarrow G \tilde{\times} X$ is $G$-equivariant. It remains to show that $\nabla_{(X, \beta)}$ is a morphism of $\mathbb{S}$-coalgebras:
\begin{align*}
\chi_X \circ G \tilde{\times}\beta \circ \nabla_{(X, \beta)} (x) &= \chi_X (h(\beta_1(x)),h(\beta_1(x))^{-1}\beta_1(x), h(\beta_1(x))^{-1}x) \\&= (\beta_1(x), h(\beta_1(x)), h(\beta_1(x))^{-1}x)
\end{align*}
\begin{align*}
	L \times \nabla_{(X, \beta)} \circ \beta(x) &= \L \times \nabla_{(X, \beta)} (\beta_1(x),x) \\
	&= (\beta_1(x), h(\beta_1(x)), h(\beta_1(x))^{-1}). 
\end{align*}
Thus  $\chi_X \circ G \tilde{\times}\beta \circ \nabla_{(X, \beta)} = L \times \nabla_{(X, \beta)} \circ \beta$. 
Next we show that $((X, \beta), \nabla_{(X, \beta)})$ is a $\mathbb{Q}$-coalgebra, that is, the following equations hold:
$$
\varepsilon^{\bQ} \circ \nabla_{(X, \beta)} = \mathrm{Id} \qquad \Delta^{\bQ} \circ \nabla_{(X, \beta)} = G \tilde{\times} \nabla_{(X, \beta)} \circ \nabla_{(X, \beta)}.
$$
It is a straightforward calculation to show that $\varepsilon^{\bQ} \circ \nabla_{(X, \beta)} = \mathrm{Id}$. Next, 
\begin{align*}
\Delta^{\bQ}_X\circ \nabla_{(X, \beta)}(x) &= \Delta^{\bQ}_X(h(\beta_1(x)), h(\beta_1(x))^{-1}x) \\&= (h(\beta_1(x)), 1, h(\beta_1(x))^{-1}x).
\end{align*}
On the other hand, 
\begin{align*}
	G\tilde{\times} \nabla_{(X, \beta)} \circ \nabla_{(X, \beta)} (x) &= G \tilde{\times} \nabla_{(X, \beta)}(h(\beta_1(x)), h(\beta_1(x))^{-1}x) \\&= (h(\beta_1(x)), 1, h(\beta_1(x))^{-1}x). 
	\end{align*}
	Lastly, we show that $(\rV (X,\beta), \rV \nabla_{(X, \beta)}:= \tilde{\nabla})$ is $\rV\bQ$-coalgebra. Consider the map
	\begin{align}
		\tilde{\nabla} : \rV (X, \beta) &\rightarrow VQ(X, \beta) \coloneq G\tilde{\times} X \label{nab1} \\
		x &\rightarrow (h(\beta_1(x)), h(\beta_1(x))^{-1}x). \nonumber
		\end{align}
		This map is $G$-equivariant as $\tilde{\nabla}(gx) = (gh(\beta_1(x)), h(\beta_1(x))^{-1}x) = g\tilde{\nabla}(x)$.
		Furthermore, 
		\begin{align*}
			\Delta^{\bQ}_X \circ \tilde{\nabla} (x) &= \Delta^\bQ_X(h(\beta_1(x)), h(\beta_1(x))^{-1}x) \\&= (h(\beta_1(x)), 1, h(\beta_1(x))^{-1}x) \\ \\
			G \tilde{\times} \tilde{\nabla} \circ \tilde{\nabla}(x) &= G\tilde{\times} \tilde{\nabla}(h(\beta_1(x)), h(\beta_1(x))^{-1}x) \\& = (h(\beta_1(x)), 1, h(\beta_1(x))^{-1}x)
			\end{align*}
	So that $G\tilde{\times}\tilde{\nabla} \circ \tilde{\nabla} = \Delta^{\bQ} \circ \tilde{\nabla}$. In addition 
	$$
	\varepsilon^{\bQ}_X \circ \tilde{\nabla}(x) = \varepsilon^{\bQ}_X(h(\beta_1(x)), h(\beta_1(x))^{-1}x) = x = \mathrm{Id}(x).
	$$
	Thus $(X,\rV\nabla_{(X, \beta)})$ is a $\rV\bQ$-coalgebra.

 Finally, applying the functor $\tilde{\rM}$ to (\ref{nab1}) we obtain 
 \begin{align*}\tilde{\rM}\tilde{\nabla} : \tilde{\rM}\rV(X, \beta) & \rightarrow \tilde{\rM}\rV\rQ(X, \beta) \\ [x] & \mapsto [(h(\beta_1(x)), h(\beta_1(x))^{-1}x)]
\end{align*} and by the above result it is straightforward to show that $(\tilde{\rM}\rV,\tilde{\rM}\tilde{\nabla})$ is a $\bQ$-opcoalgebra structure on $\tilde{\rM}\rV$.

	\end{proof}

\begin{thm}
	Let  $\tilde{\rM} : \GSET
	\rightarrow \SET$ be the
	functor which takes any
	$G$-set $X$ to the set of
	orbits in $X$, 
	$\tilde{\rM}X:=  X/G.$
	 Let $h
		: L \rightarrow G$ be a
		$G$-equivariant map and let
		$\tilde{\lambda}_N : L \times
		N \rightarrow G\tilde{\times} N$ be the
		$G$-equivariant map defined by
		$\tilde{\lambda}_N(b, n) =
		(h(b), h(b)^{-1}n).$ 
		Then the left $\chi$-coalgebra $$\lambda_N : (L \times
		N)/G \rightarrow (G\tilde{\times}
		N)/G \quad ([b,n]) \mapsto
		[\tilde{\lambda}_N(b, n)]$$ 
	corresponds bijectively
	 to $\bQ$-opcoalgebra structures on $\tilde{\rM}\rV$.
\end{thm}
\begin{proof}
	Given that $\lambda$ is left $\chi$-coalgebra we can construct the morphism $$\tilde{\rM}\tilde{\nabla} = \tilde{\rM}\tilde{\Lambda}^{-1} \circ \lambda \rV \circ \tilde{\rM}\rV\eta^{\rF^\bS \rV}$$ as follows:
	\begin{equation} \label{eqnab}
	\xymatrix{\tilde{\rM}\rV(X, \beta) \ar[rr]^{\tilde{\rM}\rV\eta^{\rF^\bS \rV}_{(X, \beta)}} & & \tilde{\rM}\rV\rF^{\bS} \rV (X,\beta) \ar[r]^{\lambda_{\rV (X, \beta)}} & \tilde{\rM}\rT\rV(X,\beta ) \ar[r]^{\tilde{\rM}\tilde{\Lambda}^{-1}_{(X, \beta)}} & \tilde{\rM}\rV\rQ(X, \beta)
			}
	\end{equation}
	
	where $\tilde{\Lambda}^{-1}$ is the colax isomorphism from the comonad $\bQ$ on $\xGSET$ to the comonad $\bT$ on $\GSET$. From equation (\ref{eqnab}) we have 
	\begin{align}
		\xymatrix{[x] \ar@{|->}[r] & [(\beta_1(x), x)] \ar@{|->}[r] & [(h(\beta_1(x)), h(\beta_1(x))^{-1}x)] \ar@{|->}[r] & [(h(\beta_1(x)), h(\beta_1(x))^{-1}x)] }.
	\end{align}
	and by Proposition \ref{corresp} $(\tilde{\rM}\rV(X, \beta),\tilde{\rM}\tilde{\nabla})$ is an $\tilde{\rM}\rV\bQ$-coalgebra. 
	
	Conversely, let $(\tilde{\rM}\rV, \tilde{\rM}\tilde{\nabla})$ be any $\bQ$-opcoalgebra on $\tilde{\rM}\rV$ and define the natural transformation $\lambda : \tilde{\rM}\rV\rF^{\bS} \Rightarrow \tilde{\rM}\rT$ with  
	$ \lambda = \tilde{\rM}\rT\varepsilon \circ \tilde{\rM}\tilde{\Lambda}F^{\bS} \circ \tilde{\rM}\tilde{\nabla}\rF^{\bS}
	$. That is, for any $G$-set $X$, we obtain
	$$
	\xymatrix{\tilde{\rM}\rV\rF^{\bS}X \ar[rr]^{\tilde{\rM}\tilde{\nabla}\rF^{\bS}X} & & \tilde{\rM}\rV\rQ\rF^\bS X \ar[r]^{\tilde{\rM}\tilde{\Lambda}_{\rF^{\bS}X}} & \tilde{\rM}\rT\rV\rF^{\bS}X \ar[r]^{\;\;\;\;\;\tilde{\rM}\rT\varepsilon_X} & \tilde{\rM}\rT X
	}
	$$ 
	
	which gives us
	\begin{equation}
		\xymatrix{[(l, x)] \ar@{|->}[r] & [(h(l), h(l)^{-1}l, h(l)^{-1}x)]) \ar@{|->}[r] & [(h(l), h(l)^{-1}l, h(l)^{-1}x)] \ar@{|->}[r] & [(h(l), h(l)^{-1}x)]
	}
			\end{equation}
			which is a left $\chi$-coalgebra by \cite[Proposition 5.5]{Boi25}.
	\end{proof}
	\section{The duplicial operator and crossed $G$-sets}

	Here we recall the simplicial set $(\bar{C}(G,\{*\}, N, d, s, t)$ in \cite[Section 3]{Boi25} which has as $n$-simplices $
	\bar{C}_n(G, \{*\}, N):=(\{*\} \times
	G^{\times(n+1)} \tilde{\times} N)/G, $ 
	with face and 
	degeneracy operators as 
	\begin{align*} 
		d_i([*, g_1, \ldots,
		g_{n+1},
		w]) & :=
		\begin{cases} 
			[*, g_2, \ldots, g_{n+1}, w] & \mbox{if } i = 0 \\
			[*,g_1, \ldots, g_{i+1}g_{i+2},
			\ldots, g_{n+1}, w] & \mbox{if } 
			0 < i < n \\  [*, g_1, \ldots ,
			g_{n}, g_{n+1}w] & \mbox{if } i = n
		\end{cases}
		\\
		s_i([*, g_1,
		\ldots, g_{n+1},w]) & 
		:= [*, g_1, \ldots,
		g_i, 1, g_{i+1}, \ldots ,
		g_{n+1}, w]
	\end{align*}
	where $[*, g_1, \ldots,
	g_{n+1}, w] \in (\{*\} \times
	G^{\times (n+1)} \tilde{\times} N)/G$.
	For the sake of convenience, we recall the following identification
	$(\{*\} \times
	G^{\times (n+1)} \tilde{\times} N)/G \cong 
	G^{\times n} \tilde{\times} N$ via  
	
	\begin{equation}\label{identi} [*, g_1, \ldots, g_{n+1}, w] \mapsto ( g_2, \ldots, g_{n+1}, w)
	\end{equation}
	with inverse map given by 
	$$
	(g_1, \ldots, g_n, w) \mapsto [*, 1, g_1, \ldots, g_n, w].
	$$

	 This simplicial set is turned into a duplicial set by equipping it with a duplicial operator
	$$ t: \bar{C}(G, \{*\}, N) \rightarrow \bar{C}(G, \{*\}, N)$$  
	$$ t_n[*, g_1, \ldots, g_{n+1},w] = [*, \alpha(w), \alpha(w)(g_2\ldots g_{n+1})^{-1}, g_2, \ldots, g_n, g_{n+1} w]$$ for $n \in \mathbb{N}$. This operator $t_n = \lambda \times \mathrm{Id}^{n} \circ \chi^{n} \circ \mathrm{Id}^{n}\times \rho$  where $\alpha(w) = (h\circ f(w))^{-1}$ in \cite[Section 5, Theorem 5.7]{Boi25} uisng the B\"ohm-\c Stefan construction \cite{BS08}. In what follows we give a more general result by expanding on a remark \cite[Remark 4.7]{Boi25} by the first author. That is we consider the case where the $G$-set $N$ is nontrivial. 
	\begin{thm}
		Assume that $N$ is a $G$-set and $L$ is a faithful $G$-set. The duplicial operator 
		$
		t : \bar{C}(G, \{*\}, N) \rightarrow \bar{C}(G, \{*\}, N) 
		$ such that for $n \in \mathbb{N}$
		$$ t_n([*, g_1, \ldots, g_{n+1}, w]) = [*, h(u), (h(u))^{-1}(g_1, \ldots, g_n), g_{n+1}w,]
		$$
		where $u = g_1\cdot\ldots\cdot g_{n+1}f(w)$
		is cyclic if and only if $(h\circ f(w))^{-1} \in \mathrm{Stab}(w)$ and $(h\circ f(gw))^{-1} = g(h\circ f(w))^{-1}g^{-1}$. 
	\end{thm}
	\begin{proof}
		Assume that the duplicial operator $t$ is cyclic then for $n = 1$,
		$$t_1^{2}([*, g_1, g_2, w]) = [*, g_1, g_2,w]. 
		$$ 
		Now, 
		\begin{align*}
			t_1^2([*,g_1, g_2, w]) &= t_1([*, h(u_2), h(u_2)^{-1}g_1, g_2w]), \qquad	\\ & =([*, h(g_1f(g_2w)), (h(g_1f(g_2w)))^{-1} h((u_2), (h\circ f(w))^{-1}w] 
		\end{align*}
		where $u_2 = g_1g_2f(w)$. Following from the above identification (\ref{identi}) we obtain 
		$$
	((h\circ f)(g_2w))^{-1}g_2 (h\circ f)(w), ((h\circ f)(w))^{-1}w)  = (g_2, w)
		$$ 
		so that $(h\circ f(w))^{-1}w = w$. That is, $(h\circ f(w))^{-1} \in \mathrm{Stab}(w)$. We also obtain
		$$(h\circ f(g_2w))^{-1} = g_2h(f(w))^{-1} g_2^{-1}.
		$$
Conversely, suppose that $(h\circ f(w))^{-1} \in \mathrm{Stab}(w)$ and $(h\circ f(gw))^{-1} = g(h\circ f(w))^{-1}g^{-1}$. Then for any $n \in \mathbb{N}$, 

\begin{align*}
&t_{n}^{n}([*, g_1, \ldots, g_{n+1}, w])
\\&= [*, h(u_2), h(u_2)^{-1}h(u_3), \ldots, h(u_{n})^{-1}h(u_{n+1}), h(u_{n+1})^{-1}g_1, g_2\cdot\ldots \cdot g_{n+1}w  ]
\end{align*}
	where $u_{2} = g_1g_2f(g_3\cdot \ldots \cdot g_{n+1}w), \ldots,  u_{n+1} = g_1\cdot \ldots \cdot g_{n+1}f(w)$.
	By applying the operator $t_{n}$ once more we obtain:
\begin{align*}
	&t_{n}([*, h(u_2), h(u_2)^{-1}h(u_3), \ldots, h(u_{n})^{-1}h(u_{n+1}), h(u_{n+1})^{-1}g_1, g_2\cdot\ldots \cdot g_{n+1}w  ])
	\\&= [*, h(u_1), h(u_1)^{-1}h(u_2), \ldots, h(u_{n})^{-1}h(u_{n+1}), (h\circ f(w))^{-1}w].
\end{align*}
where $u_1 = g_1f(g_2\cdot \ldots \cdot g_{n+1}w)$.
 Since $(h\circ f(gw))^{-1} = g(h\circ f(w))^{-1}g^{-1}$, we observe that 
\begin{align*}
	h(u_1)^{-1}h(u_2)&=h(f(g_2\cdot \ldots \cdot g_{n+1}w))^{-1}g_2h(f(g_3\cdot \ldots \cdot g_{n+1}w)) \\&= g_2h(f(g_3\cdot \ldots \cdot g_{n+1}w))^{-1}h(f(g_3\cdot \ldots \cdot g_{n+1}w)) = g_2 \\
	\vdots \\
	h(u_n)^{-1}h(u_{n+1}) &= g_{n+1}(h\circ f(w))^{-1}(h\circ f(w)) = g_{n+1}.
\end{align*}	
By the identification (\ref{identi})
$$
[*, h(u_1), h(u_1)^{-1}h(u_2), \ldots, h(u_{n})^{-1}h(u_{n+1}), (h\circ f(w))^{-1}w] = [*, g_1, g_2, \ldots, g_{n+1}, w].
$$ Thus for every $n \in \mathbb{N}$, $t_n$ is cyclic.
	\end{proof}
	
We thus see that $N$ needs to be equipped with a crossed $G$-set structure in order to have $t$ to be cyclic.

\section{The comonad $\bar{\mathbb{T}}$ }\label{comonodT}
Here we discuss an adjunction between $\SET$ and $\xGSET$. We do this by first realising that we have an $\mathbb{S}$-comonadic triangle (as used by Brzezinski in \cite{brzezinski}
$$
\xymatrix{ \SET \ar[rr]^{K} \ar@<2.5pt>[dr]^{F} & & \xGSET \ar[dl]^{V} \\ 
	& \GSET \ar@<3.5pt>[ul]^{U} & }
$$

where 
$$K : \SET \rightarrow \xGSET, \quad X \mapsto (G \tilde{\times}X, \alpha), \quad \alpha : G\tilde\times X \rightarrow L \times G\tilde{\times} X$$ is a comparison functor and 
$$
V : \xGSET \rightarrow \GSET,\quad (X, \beta) \mapsto X.
$$
Indeed $VK = F = G \tilde{\times} -$. 
Next we classify all the morphisms $G\tilde{\times} - \rightarrow L\times -$ of comonads between $G\tilde{\times}-$ and $L \times -$.
\begin{prop}
	Let $X$ be a $G$-set and let $f : X \rightarrow L$ be a set map. Then the natural transformation
	$\varphi : G \tilde{\times} - \rightarrow L \times -$ defined by 
	$\varphi_X(g,x) = (gf(x), gx)$ is a morphism of comonads and all morphisms of comonads are of this form.
\end{prop}
\begin{proof}
	It is straightforward to verify that $\varphi$ is a morphism of comonads, that is, the following diagrams 
	$$
	\xymatrix{G\tilde{\times}X \ar[dd]_{\tilde{\Delta}_X} \ar[rr]^{\varphi_X} & & L \times X \ar[dd]^{\Delta_X} \\ \\
		G\tilde{\times}G\tilde{\times} X \ar[rr]_{L\times \varphi_X \circ \varphi_{G\tilde{\times}X}} &  & L\times L \times X} 
	\quad \xymatrix{ G\tilde{\times} X \ar[ddrr]_{\tilde{\varepsilon}_X}\ar[rr]^{\varphi_X} & & L\times X \ar[dd]^{\varepsilon_X} \\ \\ & & X }
	$$
	Assume now that $\varphi_X : G\tilde{\times} X \rightarrow L \times X$ is a morphism of comonads and define $\varphi_1, \varphi_2$ by 
	$$
	\varphi_X(g,x):= (\varphi_1(g,x), \varphi_2(g, x)).
	$$
	As $\varphi_X$ is $G$-equivariant, 
	$$\varphi_X(g,x) = \varphi_X(g(1, x)) = g\varphi_X(1, x)= g(\varphi_1(1, x), \varphi_2(1, x))= (g\varphi_1(1, x), g\varphi_2(1,x))
	$$
	Thus $\varphi_1(g, x) = g\varphi_1(1, x) = gf(x)$ where $f : X \rightarrow L$ is any set map. By considering the counit compatibility condition we obtain $\varphi_2(g,x) = gx$. Thus 
	$$
	\varphi_X(g, x) = (gf(x), gx).
	$$
\end{proof}

The following is an instance of the results in \cite{brzezinski}
\begin{prop}
	Fix the categories $\SET, \GSET$, the comonad $\mathbb{S}$ on $\GSET$ and adjoint functors $G\tilde{\times}- : \SET \rightleftarrows \GSET$. There is a one-to-one correspondence between the comparison functor $\rK :\SET \rightarrow \xGSET$ in the comonadic triangle made of $L\times -$, $ G\tilde{\times}-$ and $\rU$, and the morphisms of comonads $\varphi : G\tilde{\times} - \rightarrow L \times -$.
\end{prop}
\begin{proof}
	Suppose that $\varphi : \rT \rightarrow  \rS$ is a morphism of comonads on $\GSET$. We define a natural transformation 
	$
	\mathfrak{u}: \rF \rightarrow \rS\rF
	$ such that for any set $Y$
	$$
	\mathfrak{u}_Y:G \tilde{\times}Y \rightarrow L \times G\tilde{\times}Y, \quad (g, y)\mapsto (gf(y), g, y)
	$$
	where $\mathfrak{u}_Y = \varphi_{G\tilde{\times}Y} \circ G\tilde{\times}\eta$ with $\eta$ being the unit of the adjunction between $\SET$ and $\GSET$. Then the functor $\rK : \SET \rightarrow \xGSET$ is given by $ Y \mapsto (G\tilde{\times}Y, \mathfrak{u}_Y)$. It is straightforward to verify that $\mathfrak{u}_Y$ is an $\mathbb{S}$-coalgebra map.
	Conversely given $\rK : Y \rightarrow (G\tilde{\times}Y, \beta^{G\tilde{\times}Y})$ define 
	$
	\mathfrak{u} : G\tilde{\times}- \Rightarrow L \times G\tilde{\times}$ by $\mathfrak{u}_Y = \beta^{G\tilde{\times} Y}$ where we have  $\beta^{G\tilde{\times}Y}(g, y) = (g\alpha_1(y), g, y)$. Then $\varphi : G\tilde{\times}- \Rightarrow L\times -$ is equal to the composite $\rS\tilde{\varepsilon} \circ \mathfrak{u} \rU$. That is 
	$$ \varphi_Y (g, y) = (g\alpha_1(y), gy).$$
	Thus $\varphi$ is a morphism of comonads.
\end{proof}

Since $\SET$ is complete, it has equalisers and by \cite[Proposition 1.5]{brzezinski} 
$\rK$ has a right adjoint $\rD : \xGSET \rightarrow \SET$ which sends any $\bS$-coalgebra $(X, \beta)$ to the equaliser of the pair of maps $\mathfrak{l}_X = \rU\varphi_X \circ \eta \rU (X) : X \rightarrow L \times X$ and $\beta:  X\rightarrow L \times X$. We see that $\mathfrak{l}_X(x) = \varphi_X(1, x) = (f(x), x)$. Meanwhile $\beta(x) = (\beta_1(x), x)$. Thus  
$$
\rD (X, \beta):= \{x \in X : f(x) = \beta_1(x), f, \beta_1 : X \rightarrow L  \} = \{x \in X \colon \varphi_X(g,x)  = g\beta(x)\}.
$$ 
We write down explicitly the adjunction that exists between these two functors.
\begin{prop}
	Let $\rK : \SET \rightarrow \xGSET$ be the comparison functor in the comonadic triangle made of $L \times -, G\tilde{\times}-$ and $\rU$. Then
	\begin{align*} 
	\Pi : \mathrm{Hom}_{\xGSET}(\rK Y, (X, \beta)) &\rightarrow \mathrm{Hom}_{\SET}(Y, \rD (X, \beta)) \\ \Pi(h)(y)&:= h(1,y)
	\end{align*}
	is a natural bijection with inverse 
	\begin{align*}
		\Theta: \mathrm{Hom}_{\SET}(Y, \rD (X, \beta)) &\rightarrow \mathrm{Hom}_{\xGSET}(\rK Y, (X, \beta))
		\\ \Theta(q)(g,x)&:= gq(x)
		\end{align*}
	\end{prop}
\begin{proof}
	Our goal in the first part of the proof is to show that $\Pi(h)(y)$ lands in $\rD (X, \beta)$. 
A map h: $\rK Y \coloneq (G\tilde{\times}Y, \varphi_{G\tilde{\times} Y} \circ G\tilde{\times \eta}) \rightarrow (X, \beta)$ is a  morphism of $\bS$-coalgebras. Thus $(\beta_1(h(1, y)) = \tilde{f}(1,y)$ where $\tilde{f}: G\tilde{\times} Y \rightarrow L$ is a set map with $\varphi_{G\tilde{\times}Y}(g_1, g_2, y) = (g_1\tilde{f}(g_2, y), g_1g_2, y)$.
Since $\varphi$ is a natural transformation, then for the $G$-equivariant map $h: G\tilde{\times} Y \rightarrow X$ we obtain the following commutative diagram:
$$
\xymatrix{ G\tilde{\times}G \tilde{\times} Y \ar[d]_{G\tilde{\times} h} \ar[r]^{\varphi_{G\tilde{\times} Y}} & L \times  G\tilde{\times} Y \ar[d]^{L \times h} \\
	G\tilde{\times} X \ar[r]_{\varphi_X} &  L \times X}
$$
from which we obtain $f(h(g_2, y)) = \tilde{f}(g_2, y)$. In particular, if $g_2 = 1$, then $f(h(1, y)) = \tilde{f}(1, y)$. We can now say that $h(1, y)$ is an element in the equaliser of the pair of maps $f, \beta_1$ since 
\begin{align*}
\varphi_X(g, h(1, y)) &= (gf(h(1, y)), gh(1, y)) = (g\tilde{f}(1, y), g(h(1, y))\\& = (g\beta_1(h(1, y)), gh(1,y))
= g\beta(h(1,y)). 
\end{align*}
Thus the map $\Pi$ is well defined. Next we show that the map $\Theta$ is also well defined. Indeed $\Theta(q): \rK Y \rightarrow (X, \beta)$ is $G$-equivariant as $\Theta(q)(ag, x)= agq(x) = a\Theta(q)(g, x)$. So by the naturality of $\varphi$, the following diagram commutes:
$$
\xymatrix{ G\tilde{\times}G \tilde{\times} Y \ar[d]_{G\tilde{\times} \Theta(q)} \ar[r]^{\varphi_{G\tilde{\times} Y}} & L \times  G\tilde{\times} Y \ar[d]^{L \times \Theta(q)} \\
	G\tilde{\times} X \ar[r]_{\varphi_X} &  L \times X}
$$
so that for $(g_1, g_2, y) \in G\tilde{\times} G\tilde{\times} Y$ we obtain $f(g_2q(y)) = \tilde{f}(g_2,y)$ and in particular $f(q(y)) = \tilde{f}(1, y)$. Next, $\Theta(q) : KY \rightarrow (X, \beta)$ is a morphism of $\bS$-coalgebras. Indeed $\beta \circ \Theta(q) = L \times \Theta(q) \circ \varphi_{G\tilde{\times} Y} \circ G\tilde{\times} \eta$ since 
\begin{align*}
	\beta \circ \Theta(q)(g, y) &= (g\beta_1(q(y), gq(y)) = g\beta(q(y)) \\&= \varphi_X(g, q(y)) = (gf(q(y), gq(y)))
	\\ L \times \Theta(q) \circ \varphi_{G\tilde{\times} Y} \circ G\tilde{\times} \eta (g, y) &= L \times \Theta(q) \circ \varphi_{G\tilde{\times} Y}(g, 1, y) \\&= (g\tilde{f}(1, y), gq(y)) = (gf(q(y)), gq(y)).
	\end{align*}
	Finally, 
	\begin{align*}
		(\Theta \circ \Pi(h))(g, y) &= g\Pi(h)(y) = gh(1, y) = h(g, y) \\ 
		(\Pi \circ \Theta(q))(y) & = \Theta(q)(1, y) = q(y).
	\end{align*}	
	
	Thus $\Theta \circ \Pi = \mathrm{Id}_{\mathrm{Hom}_{\xGSET}}$ and $\Pi \circ \Theta = \mathrm{Id}_{\mathrm{Hom}_{\SET}}$.
\end{proof}
Since the unit and counit of the  adjunction $\rK \colon \xGSET \rightleftarrows \SET \colon \rD$ are not isomorphisms the above theory does not apply to Corollary 3.3 in \cite{BBK25}.
	
	\bibliographystyle{plain}
		\bibliography{BTwu.bib}

@article{Boi25,
 author = {Boiquaye, John},
 title = {Cyclic sets from the nerve of a group and the {B{\"o}hm}-{{\c{S}}tefan} construction},
 fjournal = {Theory and Applications of Categories},
 journal = {Theory Appl. Categ.},
 issn = {1201-561X},
 volume = {44},
 pages = {181--195},
 year = {2025},
 language = {English},
 url = {www.tac.mta.ca/tac/volumes/44/5/44-05abs.html},
 zbMATH = {7977824},
 Zbl = {1570.18010}
}

@misc{BBK25,
 author = {Ivan Bartulovi{\'c} and John Boiquaye and Ulrich Kr{\"a}hmer},
 title = {Duplicial functors, descent categories and generalized {Hopf} modules},
 year = {2025},
 howpublished = {Preprint, {arXiv}:2501.14561 [math.{CT}] (2025)},
 url = {https://arxiv.org/abs/2501.14561},
 arXiv = {arXiv:2501.14561}
}

@article{BS08,
 author = {B{\"o}hm, Gabriella and {\c{S}}tefan, Drago{\c{s}}},
 title = {(Co)cyclic (co)homology of bialgebroids: {An} approach via (co)monads},
 fjournal = {Communications in Mathematical Physics},
 journal = {Commun. Math. Phys.},
 issn = {0010-3616},
 volume = {282},
 number = {1},
 pages = {239--286},
 year = {2008},
 language = {English},
 doi = {10.1007/s00220-008-0540-3},
 zbMATH = {5373480},
 Zbl = {1153.18004}
}

@article{KKS15,
 author = {Kowalzig, Niels and Kr{\"a}hmer, Ulrich and Slevin, Paul},
 title = {Cyclic homology arising from adjunctions},
 fjournal = {Theory and Applications of Categories},
 journal = {Theory Appl. Categ.},
 issn = {1201-561X},
 volume = {30},
 pages = {1067--1095},
 year = {2015},
 language = {English},
 zbMATH = {6520703},
 Zbl = {1344.18003}
}

@book{Loday,
 author = {Loday, Jean-Louis},
 title = {Cyclic homology.},
 edition = {2nd ed.},
 fseries = {Grundlehren der Mathematischen Wissenschaften},
 series = {Grundlehren Math. Wiss.},
 issn = {0072-7830},
 volume = {301},
 isbn = {3-540-63074-0},
 year = {1998},
 publisher = {Berlin: Springer},
 language = {English},
 zbMATH = {1093754},
 Zbl = {0885.18007}
}

@misc{beck,
 author = {Beck, Jon},
 title = {Distributive laws},
 year = {1969},
 language = {English},
 howpublished = {Semin. {Triples} categor. {Homology} {Theory}, {ETH} 1966/67, {Lect}. {Notes} {Math}. 80, 119-140 (1969).},
 doi = {10.1007/bfb0083084},
 zbMATH = {3296291},
 Zbl = {0186.02902}
}

@article{brzezinski,
  title={Galois structures},
  author={Brzezinski, Tomasz and Janelidze, George and Maszczyk, Tomasz},
  journal={Lecture Notes on Noncommutative Geometry and Quantum Groups, Editor PM Hajac, notes by P. Witkowski, available at http://www. mimuw. edu. pl/\~{} pwit/toknotes/toknotes. pdf},
  year={2008}
}

@article{FreydYetter,
 author = {Freyd, Peter J. and Yetter, David N.},
 title = {Braided compact closed categories with applications to low dimensional topology},
 fjournal = {Advances in Mathematics},
 journal = {Adv. Math.},
 issn = {0001-8708},
 volume = {77},
 number = {2},
 pages = {156--182},
 year = {1989},
 language = {English},
 doi = {10.1016/0001-8708(89)90018-2},
 zbMATH = {4113553},
 Zbl = {0679.57003}
}

@incollection{kassel,
 author = {Kassel, Ch.},
 title = {Quantum groups},
 booktitle = {Algebra and operator theory. Proceedings of the colloquium, Tashkent, Uzbekistan, September 29--October 5, 1997},
 isbn = {0-7923-5094-4},
 pages = {213--236},
 year = {1998},
 publisher = {Dordrecht: Kluwer Academic Publishers},
 language = {Spanish},
 zbMATH = {1304483},
 Zbl = {1113.17304}
}

@article{Yoshida,
 author = {Yoshida, Tomoyuki},
 title = {Crossed {{\(G\)}}-sets and crossed {Burnside} rings},
 fjournal = {RIMS Kokyuroku},
 journal = {RIMS Kokyuroku},
 issn = {1880-2818},
 volume = {991},
 pages = {1--15},
 year = {1997},
 language = {English},
 zbMATH = {1373448},
 Zbl = {0933.19001}
}

@article{Wol73,
 author = {Wolff, Harvey},
 title = {V-localizations and {V}-monads. {II}},
 fjournal = {Pacific Journal of Mathematics},
 journal = {Pac. J. Math.},
 issn = {1945-5844},
 volume = {63},
 pages = {579--589},
 year = {1976},
 language = {English},
 doi = {10.2140/pjm.1976.63.579},
 zbMATH = {3539406},
 Zbl = {0346.18011}
}

@book{Weibel,
 author = {Weibel, Charles A.},
 title = {An introduction to homological algebra. 1st pbk-ed},
 edition = {1st pbk-ed.},
 fseries = {Cambridge Studies in Advanced Mathematics},
 series = {Camb. Stud. Adv. Math.},
 volume = {38},
 isbn = {0-521-55987-1},
 year = {1995},
 publisher = {Cambridge: Cambridge Univ. Press},
 language = {English},
 zbMATH = {841435},
 Zbl = {0834.18001}
}

%
	
\end{document}